\documentclass[12pt,a4paper]{amsart}

\usepackage[margin=1in]{geometry}
 
\usepackage{graphicx,amssymb}
\usepackage{mathrsfs}
\usepackage{amssymb,amsmath,amsthm,color}
\usepackage{graphicx}
\usepackage{hyperref}
\usepackage{url} 
\usepackage{setspace}

\numberwithin{equation}{section}

\newcommand{\R}{\mathbb R}

\newcommand{\T}{\mathbb T}

\def\TagOnRight

\def\R {\mathbb{R}}

\newcommand{\be}{\begin{equation}}
\newcommand{\ee}{\end{equation}}
\newcommand{\bea}{\begin{eqnarray}}
\newcommand{\eea}{\end{eqnarray}}
\newcommand{\Bea}{\begin{eqnarray*}}
\newcommand{\Eea}{\end{eqnarray*}}
\newcommand{\bt}{\begin{Theorem}}
\newcommand{\et}{\end{Theorem}}

\newcommand{\bpr}{\begin{Proposition}}
\newcommand{\epr}{\end{Proposition}}

\newcommand{\bl}{\begin{Lemma}}
\newcommand{\el}{\end{Lemma}}
\newcommand{\bi}{\begin{itemize}}
\newcommand{\ei}{\end{itemize}}

\newtheorem{Definition}{Definition}[section]
\newtheorem{Theorem}[Definition]{Theorem}
\newtheorem{Lemma}[Definition]{Lemma}
\newtheorem{Proposition}[Definition]{Proposition}
\newtheorem{Corollary}[Definition]{Corollary}
\newtheorem{Remark}[Definition]{Remark}

\usepackage{mathrsfs}
\usepackage{amsthm}
\usepackage{hyperref}
\usepackage{comment}

\begin{document}
\baselineskip16pt

\title[Composition operators on $W^{p,q}$  spaces]{Composition operators on  Wiener amalgam  spaces}
\author{Divyang G. Bhimani}
\address {Department of Mathematics, University of Maryland, College Park, MD 20742}
 \email {dbhimani@math.umd.edu}
\maketitle
\begin{abstract}
For  a complex function $F$ on $\mathbb C$, we study the associated composition operator $T_{F}(f):=F\circ f= F(f)$ on  Wiener amalgam  $W^{p,q}(\mathbb R^d) \ (1\leq p< \infty, 1\leq q<2).$ We have shown $T_{F} $ maps $W^{p, 1}(\mathbb R^d)$ to $W^{p,q}(\mathbb R^d)$ if and only if  $F$ is real analytic on $\mathbb R^2$ and $F(0)=0.$  Similar result is proved in the case of modulation spaces $M^{p,q}(\mathbb R^d).$ In particular, this gives an affirmative answer to the open question proposed in \cite[p.646]{dgb-pkr}
\end{abstract}

\section{Introduction}
Let $X$ and $Y$ be  normed spaces of complex functions on $\mathbb R^d.$ For a given function $F:\mathbb R^{2}(=\mathbb C)\to \mathbb C,$ we  associate it, with the composition operator $T_{F}:f\mapsto F(f),$ where $F(f)= F\circ f$ is the composition of functions $F$ and $f:\mathbb R^d \to \mathbb C$. If $T_{F}(X) \subset Y,$ we say the composition operator $T_{F}$ takes $X$ to $Y.$ In particular, if $T_{F}(X)\subset X,$ we say the composition operator $T_{F}$ acts on $X.$  Which functions $F$ have the property that the composition operator $T_{F}$ takes $X$ to $Y$? Of course, the properties  of the operator $T_{F}$ strongly depend on $X$ and $Y.$ The aim of this paper is to take a small step toward the answer in the case of modulation and Wiener amalgam spaces (see Subsection \ref{mw}  for precise definitions).

In the last decade,  modulation  and Wiener amalgam  spaces   have turned out to be very fruitful  within pure and applied mathematics.  In fact, these spaces are nowadays present in investigations that concern problems on pseudo differential/Fourier integral operators, Strichartz estimates, and so on (we refer the reader to recent survey \cite{rsw} and the reference therein). For instance: the unimodular Fourier multiplier operator $e^{i|D|^{\alpha}}$ is not bounded on  most of the Lebesgue spaces $L^{p}(\mathbb R^{d}) \ (p\neq 2)$, in contrast it is bounded on $W^{p,q}(\mathbb R^{d})( 1\leq p, q \leq \infty)$  for $\alpha \in [0,1]$, and on $M^{p,q}(\mathbb R^{d}) (1\leq p,q \leq \infty)$ for  $\alpha \in [0,2]$ (cf. \cite{ambenyi,benyi, J1}). The cases $\alpha =1, 2$ are of particular interest because they occur in the time evolution of wave  and Schr\"odinger equations respectively.

Many mathematicians have been using these spaces as a regularity class of initial data of the Cauchy problem for nonlinear  dispersive equations (cf. \cite{bao,baob, ambenyi, benyi, E1}). 
But one of the underneath  issue  of the nonlinear dispersive equations  in the realm of  modulation and Wiener amalgam spaces is to determine,  which is the most general nonliearity one can take, which is not yet completely clear. Composition operators are simple examples of nonlinear mappings.  And when we try to study local and global well-posedness results for nonlinear dispersive equations (Schr\"odinger/wave/Klein-Gordon  etc.)  in modulation and Wiener amalgam spaces,  it is indispensable to study nonlinear mappings on it.  Taking these consideration into our account,  we are motivated to study composition operators on these spaces. Specifically, we prove following
\begin{Theorem}
\label{MT}
Let $1\leq p \leq  \infty, 1\leq  q <2,$ and  pairs 
$$(X,Y)= (M^{p,1}(\mathbb R^d), M^{p,q}(\mathbb R^d)) \ \text{or}  \   (W^{p,1}(\mathbb R^d), W^{p,q}(\mathbb R^d)).$$ Suppose that $T_{F}$ is the composition operator associated to a complex function $F$ on $\mathbb C.$   Then
\begin{enumerate}
\item \label{nc1}   If  $T_{F}$ takes $X$ to $Y,$ then $F$ must be real analytic on $\mathbb R^{2}.$ Moreover, $F(0)=0$ if $p<\infty.$
\item  \label{sc} If $F$ is real analytic on $\mathbb R^{2}$ which takes origin to itself, and $p<\infty,$ then $T_{F}$  acts on $X.$
\end{enumerate}
\end{Theorem}

\begin{Corollary} Let $1\leq p <  \infty, 1\leq  q <2,$ and  pairs 
$$(X,Y)= (M^{p,1}(\mathbb R^d), M^{p,q}(\mathbb R^d)) \ \text{or}  \   (W^{p,1}(\mathbb R^d), W^{p,q}(\mathbb R^d)).$$Then $F$ is real analytic on $\mathbb R^2$ and $F(0)=0$ if and only if $T_{F}$  takes $X$ to $Y.$ In particular, $F$ is real analytic on $\mathbb R^2$ and $F(0)=0$ if and only if $T_{F}$ acts on $X.$
\end{Corollary}
\begin{Corollary} \label{op}
 (1)There exists $f\in W^{p,1} (\R^d)$ such that $f|f|^{\alpha} \notin W^{p,q}(\R^d)$ $ 
 (1\leq p \leq \infty, 1\leq q <2)$, for any $\alpha \in (0, \infty) \setminus  2\mathbb N.$   (2)There exists $f\in M^{p,1} (\R^d)$ such that $f|f|^{\alpha} \notin M^{p,q}(\R^d)$ $ 
 (1\leq p \leq \infty, 1\leq q <2)$, for any $\alpha \in (0, \infty) \setminus  2\mathbb N.$ 
\end{Corollary}
We note that  recently Bhimani-Ratnakumar \cite[Theorem 3.9]{dgb-pkr} have proved  if $F:\mathbb R^2\to \mathbb C$ is real analytic and $F(0)=0$, then $T_{F}$ acts on  $X=M^{1,1}(\mathbb R^d), d\geq 1,$ and this  result is generalized by Kobayashi-Sato \cite[Theorem 1.1]{ks} for $X=M^{p,1}(\mathbb R)$ or $W^{p,1}(\mathbb R)$ for $1<p<\infty,$ although it is restricted to the case when $d=1.$   Thus we remark that  our Theorem \ref{MT} \eqref{sc} settles the case when  $d>1,$ and gives an affirmative answer to the open question proposed in \cite[p.646]{dgb-pkr}.  It is  also worth noting following:
\begin{Remark}\noindent
\begin{enumerate}
\item  In  Theorem \ref{MT}\eqref{nc1} conditions on range of $q\in[1, 2)$ is sharp in the sense that if we take $q=2$ then the same conclusion may not hold. (Since $W^{2,2}(\mathbb R^d)= L^2(\mathbb R^d),$ if $T_F: W^{p,1}(\mathbb R^d) \to W^{p,2}(\mathbb R^d),$ then $F$ may not be real analytic on $\mathbb R^2$.) 

\item Theorem 3.2  of Bhimani-Ratnakumar \cite{dgb-pkr} is a particular case of Theorem 1.2 (1) as $M^{p,1}(\mathbb R^{d})$ is a proper subclass of $M^{p,q}(\mathbb R^{d}) (1<q<2).$

\item  In view of Corollary \ref{op} we can point out, the standard method for evolving nonlinear dispersive equations, with the nonlinearity $f|f|^{\alpha} \ (\alpha \in (0, \infty) \setminus  2\mathbb N),$ which is of great importance in application, is ruled out.\\  
 
\end{enumerate}
\end{Remark}

The sequel contains  required notations and  preliminary  in Section 2, proof for the necessary condition
of Theorem \ref{MT} \eqref{nc1}  in Section 3, proof for the sufficient condition of Theorem \ref{MT} \eqref{sc}, and concluding remarks in Section 4.
\section{Notations and Preliminaries}  
\subsection{Notations} The notation $A \lesssim B $ means $A \leq cB$ for a some constant $c > 0 $, whereas $ A \asymp B $ means $c^{-1}A\leq B\leq cA $ for some $c\geq 1$. The symbol $A_{1}\hookrightarrow A_{2}$ denotes the continuous embedding  of the topological linear space $A_{1}$ into $A_{2}.$ The mixed  $L^{p,q}(\mathbb R^{d}\times \mathbb R^{d})$ norm is denoted by 
 $$\|f\|_{L^{p,q}}=\left( \int_{\mathbb R^{d}} \left ( \int_{\mathbb R^{d}} |f(x,w)|^{p} dx \right)^{q/p} dw\right )^{1/q} \ \  (1\leq p, q < \infty),$$
the $L^{\infty}(\mathbb R^{d})$ norm  is $\|f\|_{L^{\infty}}= \text{ess.sup}_{ x\in \mathbb R^{d}}|f(x)| $,  the $\ell^{q}(\mathbb Z^{d})$ norm is $\|a_{n}\|_{\ell^{q}}= \left(\sum_{n\in \mathbb Z^{d}} |a_{n}|^{q}\right )^{1/q}$. We denote $d-$dimensional torus by $\mathbb T^d \equiv [0, 2\pi)^d,$ and $L^{p}(\mathbb T^d)-$norm is denoted by
$$\|f\|_{L^{p}(\mathbb T^d)}=\left( \int_{[0, 2\pi)^d} |f(t)|^p dt \right)^{1/p}.$$
The space of smooth functions on $\mathbb R^d$ with compact support is denoted by $C_{c}^{\infty}(\mathbb R^d)$,  the Schwartz class is $\mathcal{S}(\mathbb R^{d})$ (with its usual topology), the space of tempered distributions is $\mathcal{S'}(\mathbb R^{d}).$ For $x=(x_1,\cdots, x_d), y=(y_1,\cdots, y_d) \in \mathbb R^d, $ we put $x\cdot y = \sum_{i=1}^{d} x_i y_i.$
Let $\mathcal{F}:\mathcal{S}(\mathbb R^{d})\to \mathcal{S}(\mathbb R^{d})$ be the Fourier transform  defined by  
\begin{eqnarray}
\mathcal{F}f(w)=\widehat{f}(w)=\int_{\mathbb R^{d}} f(t) e^{- 2\pi i t\cdot w}dt, \  w\in \mathbb R^d.
\end{eqnarray}
Then $\mathcal{F}$ is a bijection  and the inverse Fourier transform  is given by
\begin{eqnarray}
\mathcal{F}^{-1}f(x)=f^{\vee}(x)=\int_{\mathbb R^{d}} f(w)\, e^{2\pi i x\cdot w} dw,~~x\in \mathbb R^{d},
\end{eqnarray}
and this Fourier transform can be uniquely extended to $\mathcal{F}:\mathcal{S}'(\mathbb R^d) \to \mathcal{S}'(\mathbb R^d).$   For $s\in \mathbb R, w\in \mathbb R^{d},$ we put $ \langle w \rangle^{s} = (1+|w|^{2})^{s/2}.$

\subsection{Modulation and Wiener amalgam spaces}\label{mw}
Let $g\in \mathcal{S}(\mathbb R^{d})$ be a non-zero window function. The short-time Fourier transform  (STFT) of a function(tempered distribution) $f$ with respect to a  window window  $g$ is 
\begin{eqnarray}\label{stft}
V_{g}f(x,w):= \int_{\mathbb R^{d}} f(t) \overline{g(t-x)} e^{-2\pi i w\cdot t}dt,  \  (x, w) \in \mathbb R^{2d}
\end{eqnarray}
 whenever the integral exists.
 
 In 1983 Feichtinger \cite{HG4} introduced a class of Banach spaces, which allow a
measurement of space variable and Fourier transform variable of a function or
distribution $f$ on $\mathbb R^d$ simultaneously using the STFT, the so-called modulation spaces.
\begin{Definition}[modulation spaces]
 \label{df2}
 For $1\leq p, q \leq \infty,$ and  for given a non zero smooth rapidly decreasing function $g\in \mathcal{S}(\mathbb R^{d})$, the weighted modulation space $M^{p,q}_{s} (\mathbb R^{d})$ consists  of all tempered distributions  $f\in \mathcal{S}'(\mathbb R^{d})$ for which,  the following norm 
$$\|f\|_{M^{p,q}_{s}}=\left(\int_{\mathbb R^{d}}\left(\int_{\mathbb R^{d}} |V_{g}f(x,w)|^{p} dx\right)^{q/p} \langle w \rangle^{sq}dw\right)^{1/q}$$
is finite, with the usual modification if $p$ or $q$ are infinite.
 \end{Definition}
 This definition is independent of the choice of the window, in the sense that different window functions yield equivalent modulation space norms (cf. \cite[Proposition 11.3.2 (c), p.233]{gro}). When $s=0,$ we simply write $ M^{p,q}_{0}(\mathbb R^{d})=M^{p,q}(\mathbb R^{d}).$\\
By reversing the order of integration we define the another family of spaces, so-called Wiener amalgam spaces.

\begin{Definition}[Wiener amalgam spaces]
\label{df1}
For $1\leq p, q \leq \infty, s\in \mathbb R$ and $0\neq g\in \mathcal{S}(\mathbb R^{d})$, the weighted Wiener amalgam space $W_{s}^{p,q}(\mathbb R^{d})$ consists of all tempered distributions $f\in \mathcal{S'}(\mathbb R^{d})$  such that the norm
\begin{eqnarray}
\|f\|_{W^{p,q}_{s}}=\left(\int_{\mathbb R^{d}} \left(\int_{\mathbb R^{d}}|V_{g}f(x,w)|^{q}\langle w \rangle^{sq} dw \right)^{p/q} dx \right)^{1/p}\nonumber
\end{eqnarray}
is finite, with usual modifications if $p$ or $q=\infty$.
\end{Definition}
This definition is independent of the choice of the window $g$, in the sense that different window functions yields equivalent Wiener amalgam  space norms. When $s=0,$ we simply write $W^{p,q}_{0}(\mathbb R^{d})=W^{p,q}(\mathbb R^{d}).$\\

We note that there is another characterization \cite{tribel} of Wiener amalgam and modulation spaces: let $\phi \in \mathcal{S}(\mathbb R^{d})$ such that 
$$\text{supp} \phi \subset (-1,1)^{d}$$ and $$\sum_{k\in \mathbb Z^{d}} \phi (w-k)=1, \forall w \in \mathbb R^{d}.$$Then we have the equivalence 
$$\|f\|_{W_{s}^{p,q}}\asymp \| \|\langle k \rangle^{s} \phi (D-k)f \|_{\ell^{q}} \|_{L^{p}},$$ 
and
$$\|f\|_{M_{s}^{p,q}}\asymp \| \|\langle k \rangle^{s} \phi (D-k)f \|_{L^{p}} \|_{\ell^{q}},$$ 
where $\phi(D-k)f= \mathcal{F}^{-1}(\widehat{f}\cdot T_{k}\phi).$ \\
 \subsection{Properties of modulation and Wiener amalgam spaces}
We gather some basic properties of Wiener amalgam  and modulation spaces which will be frequently used in the sequel. 
\begin{Lemma}\label{pl} Let $p,q, p_{i}, q_{i}\in [1, \infty]$  $(i=1,2).$
\begin{enumerate}
\item \label{sce} $\mathcal{S}(\mathbb R^{d}) \hookrightarrow W^{p,q }(\mathbb R^{d}) \hookrightarrow \mathcal{S'}(\mathbb R^{d})$ and $\mathcal{S}(\mathbb R^{d}) \hookrightarrow M^{p,q }(\mathbb R^{d}) \hookrightarrow \mathcal{S'}(\mathbb R^{d}).$
\item \label{tz} $W^{p,q_{1}}(\mathbb R^{d}) \hookrightarrow L^{p}(\mathbb R^{d}) \hookrightarrow W^{p,q_{2}}(\mathbb R^{d})$ and $M^{p,q_{1}}(\mathbb R^{d}) \hookrightarrow L^{p}(\mathbb R^{d}) \hookrightarrow M^{p,q_{2}}(\mathbb R^{d})$ holds for $q_{1}\leq \text{min} \{p, p'\}$ and $q_{2}\geq \text{max} \{p, p'\}$ with $\frac{1}{p}+\frac{1}{p'}=1.$
\item \label{incl} If $q_{1}\leq q_{2}$ and $p_{1}\leq p_{2}$, then $W^{p_{1}, q_{1}}(\mathbb R^{d}) \hookrightarrow W^{p_{2}, q_{2}}(\mathbb R^{d})$ and $M^{p_{1}, q_{1}}(\mathbb R^{d}) \hookrightarrow M^{p_{2}, q_{2}}(\mathbb R^{d}).$
\item \label{rmw}$M^{p,q}(\mathbb R^{d}) \hookrightarrow W^{p,q}(\mathbb R^{d})$ when $q\leq p$ and $W^{p,q}(\mathbb R^{d}) \hookrightarrow M^{p,q}(\mathbb R^{d})$ when $p\leq q.$
\item  \label{fti} The space $W^{p,p}(\mathbb R^{d})=M^{p,p}(\mathbb R^{d}) (1\leq p < \infty)$ is invariant under Fourier transform.
\item \label{bsmw} The spaces $W^{p,q}(\mathbb R^{d})$ and $M^{p,q}(\mathbb R^{d})$ are Banach space.
\label{uf}\item The spaces  $W^{p,q}(\mathbb R^{d})$ and $M^{p,q}(\mathbb R^{d})$ are invariant under complex conjugation.

\item \label{da}$\mathcal{S}(\mathbb R^d)$ is dense in $M^{p,q}(\mathbb R^d)$ and $W^{p,q}(\mathbb R^d)$ for $p,q \in [1, \infty).$
\item \label{dp} Let  $0<\lambda \leq 1,$ and put $f_{\lambda}(x)=f(\lambda x).$ There exists constants  $C$ and $C'$ such that  $\|f_{\lambda}\|_{M^{\infty,1}} \leq C \|f\|_{M^{\infty,1}}$ for $f\in M^{\infty,1}(\mathbb R^d).$ 
\end{enumerate}
\end{Lemma}
\begin{proof}
The proof of statements  \eqref{tz} and  \eqref{incl} can be found in \cite{nonom} and \cite{moT} respectively.  For the proof of statement \eqref{dp}, see \cite{s1}.  For the  proof of statements \eqref{incl}, \eqref{bsmw}, and \eqref{da},  see \cite{gro}. We only give the arguments for the statement (6) because it provide the reader with some insight about the fundamental identity of time-frequency analysis: in fact, easy computation gives the fundamental identity
$$ V_gf(x, w) = e^{-2 \pi i x \cdot w } \, V_{\widehat{g}} \widehat{f}(w, -x)$$
but this immediately gives a proof of (6). 
\end{proof}

\begin{Proposition}[Algebra property]\label{ap} Let $p,q, p_{i}, q_{i}\in [1, \infty]$  $(i=0,1,2)$  satisfy
$\frac{1}{p_1}+ \frac{1}{p_2}= \frac{1}{p_0}$ and $\frac{1}{q_1}+\frac{1}{q_2}=1+\frac{1}{q_0}. $ Then
\begin{enumerate}
\item $M^{p_1, q_1}(\mathbb R^{d}) \cdot M^{p_{2}, q_{2}}(\mathbb R^{d}) \hookrightarrow M^{p_0, q_0}(\mathbb R^{d})$ with norm inequality $$\|f g\|_{M^{p_0, q_0}}\lesssim \|f\|_{M^{p_1, q_1}} \|g\|_{M^{p_2,q_2}}.$$
In particular, $M^{p,1}(\mathbb R^d)$ is an algebra under pointwise multiplication with norm inequality
$$\|fg\|_{M^{p,1}} \lesssim \|f\|_{M^{p,1}} \|g\|_{M^{p,1}}.$$
\item \label{aws} $W^{p_1, q_1}(\mathbb R^{d}) \cdot W^{p_{2}, q_{2}}(\mathbb R^{d}) \hookrightarrow W^{p_0, q_0}(\mathbb R^{d})$ with  norm inequality $$\|f g\|_{W^{p_0, q_0}}\lesssim \|f\|_{W^{p_1, q_1}} \|g\|_{W^{p_2,q_2}}.$$In particular, $W^{p,1}(\mathbb R^d)$ is an algebra under pointwise multiplication with norm inequality
$$\|fg\|_{W^{p,1}} \lesssim \|f\|_{W^{p,1}} \|g\|_{W^{p,1}}.$$
\end{enumerate}
\end{Proposition}
\begin{proof}
 cf. \cite{bao}, \cite[Corollary 2.7]{ambenyi}, and  \cite[Lemma 2.2]{E1}.
\end{proof}
We refer to \cite{gro} for a classical foundation of these spaces and \cite{baob} for some recent developments in PDEs for these spaces and the references therein. 

\section{necessary condition}
In this section, we prove Theorem \ref{MT} \eqref{nc1}: if the composition operator  $T_{F}$  takes Wiener amalgam spaces $W^{p,1}(\mathbb R^{d})$ to $W^{p,q}(\mathbb R^{d}),$ then, necessarily, $F$ is real analytic on $\mathbb R^{2}.$ And also a similar necessity condition for modulation spaces. We start with following:
\begin{Definition}
\label{red}
A complex valued function $F,$ defined on  an open set  $E$ in the plane $\mathbb R^{2}$, is said to be real analytic on $E$, if to every point $(s_{0}, t_{0}) \in E,$ there corresponds an expansion of the form
$$F(s, t)= \sum_{m,n=0}^{\infty} a_{mn} \, (s-s_{0})^{m} \, (t-t_{0})^{n}, \hskip.1in a_{mn} \in \mathbb C$$ 
which converges absolutely for all $(s,t)$ in some neighbourhood of $(s_{0}, t_{0}).$
If $E=\mathbb R^{2}$ and the above series converges absolutely for all $(s,t) \in \mathbb R^2$, then $F$ is called real entire.
\end{Definition}
We let $A^{q}(\mathbb T^{d})$ be the class of all complex functions $f$ on the $d-$torus $\mathbb T^d$ whose Fourier coefficients
 $$\widehat{f}(m)=\int_{\mathbb T^d}f(x)e^{-2\pi i m\cdot x} dx,   \ (m\in \mathbb Z^{d})$$
satisfy the condition
$$\|f\|_{A^{q}(\mathbb T^{d})}:=\|\widehat{f}\|_{\ell^{q}}<\infty.$$

We recall, the classical theorem of Katznelson \cite[p.156]{rudin}, see also, \cite[Theorem 6.9.2]{rb} for $A^{1}(\mathbb T)$ which have been proved in 1959, and later generalized by Rudin \cite {W1} in 1962 for $A^{q}(G)$, where $G$ is infinite compact  abelian  group and $1<q<2.$ We just rephrased it here by combining both of them  as required in our context.
\begin{Theorem}[ Katznelson-Rudin ]
\label{HKKR}
Suppose that $T_{F}$ is the composition operator associated to a complex function $F$ on $\mathbb C,$ and $1\leq q<2.$
If $T_{F}$ takes $A^{1}(\mathbb T^{d})$ to $A^{q}(\mathbb T^{d}),$ then $F$ is real analytic on $\mathbb R^{2}.$
\end{Theorem}
Now we introduce periodic Wiener amalgam  and modulation spaces, and for this reason, first we recall some definitions, and introduce temporary notations. We are starting by noting that there is a one-to-one correspondence  between functions on $\mathbb R^{d}$ that are 1-periodic in each of the coordinate directions and functions on torus $\mathbb T^d,$ and we may identify $\mathbb T^{d}=\mathbb R^{d}/\mathbb Z^{d}$ with $[0, 1)^{d}.$ Let $\mathcal{D}(\mathbb T^d)$ be the vector space  $C^{\infty}(\mathbb T^{d})$ endowed with the usual test function topology, and let $\mathcal{D}'(\mathbb T^{d})$ be its dual, the space of distributions on $\mathbb T^{d}$. Let $\mathcal{S}(\mathbb Z^{d})$ denote the space of rapidly decaying functions $\mathbb Z^{d} \to \mathbb C.$ Let $\mathcal{F}_T:\mathcal{D}(\mathbb T^{d}) \to \mathcal{S}(\mathbb Z^{d})$  be the toroidal Fourier transform
(hence the subscript $T$ ) defined by 
$$(\mathcal{F}_Tf)(\xi):=\hat{f}(\xi)= \int_{\mathbb T^d}f(x)e^{-2\pi i \xi \cdot x} dx , \  \ (\xi \in \mathbb Z^{d}).$$
Then $\mathcal{F}_T$ is a bijection and  the inverse Fourier transform is given by
$$ (\mathcal{F}^{-1}_Tf)(x):= \sum_{\xi\in \mathbb Z^{d}} \hat{f}(\xi) e^{2\pi i \xi \cdot x}, \   \ (x\in \mathbb T^{d}),$$
and this Fourier transform is extended uniquely to $\mathcal{F}_T:\mathcal{D}'(\mathbb T^d) \to \mathcal{S}'(\mathbb Z^d)$. \\

The Wiener amalgam spaces $W^{p,q}(\mathbb T^{d})$ consists of all  $f\in \mathcal{D}'(\mathbb T^{d})$ such that 
$$\|f\|_{W^{p,q}(\mathbb T^{d})}:=\| \| \phi(D_T-k)f \|_{\ell^{q}} \|_{L^{p}(\mathbb T^{d})}<\infty,$$
and modulation spaces  $M^{p,q}(\mathbb T^d)$ consists of all $f\in \mathcal{D}'(\mathbb T^d)$
such that
$$\|f\|_{M^{p,q}(\mathbb T^{d})}:=\| \| \phi(D_T-k)f \|_{L^{p}(\mathbb T^{d})} \|_{\ell^{q}}<\infty,$$
for some $\phi$ with compact support in the discrete topology of $\mathbb Z^{d}$ ,  where $\phi(D_T-k)f= \mathcal{F}_T^{-1}\left(T_{k}\phi \cdot  \mathcal{F}_Tf\right).$

Next result ensures that Wiener amalgam and modulation spaces coincides with the classical Fourier algebra. Specifically, we have
\begin{Proposition} \label{B1}
Let $1\leq p , q \leq \infty$. Then, we have  
$$M^{p,q}(\mathbb T^{d})=W^{p,q} (\mathbb T^{d})= A^{q}(\mathbb T^{d}),$$ 
with norm inequality
$$\|f\|_{M^{p,q}(\mathbb T^d)}\asymp \|f\|_{W^{p,q}(\mathbb T^d)} \asymp \|f\|_{A^{q}(\mathbb T^d)}.$$
\end{Proposition}
\begin{proof}
 For the proof we refer to  \cite[Section 5]{moT} and \cite[Lemma 1]{ko}.
\end{proof}

We now define the local-in-time versions of the  Wiener amalgam  and modulation spaces  in the following way. Given an interval $I=[0, 1)^{d}$  we let $W^{p,q}(I)$ the restriction of $W^{p,q}(\mathbb R^{d})$ onto $I$ via 
\begin{eqnarray}\label{litwd}
\|f\|_{W^{p,q}(I)}:=\text{inf} \{\|g\|_{W^{p,q}(\mathbb R^{d})}:g=f \ \text{on} \  I \},
\end{eqnarray}
and $M^{p,q}(I)$ the restriction of $M^{p,q}(\mathbb R^{d})$ onto $I$ via
\begin{eqnarray}
\|f\|_{M^{p,q}(I)}=\text{inf} \{\|g\|_{M^{p,q}(\mathbb R^{d})}:g=f \ \text{on} \  I \}.
\end{eqnarray}

We note that B\'enyi-Oh  has proved the  ``equivalence" of the periodic function spaces ( $M^{p,q}(\mathbb T^d)$ and $W^{p,q}(\mathbb T^d)$) and their local-in-time versions (defined on  a bounded interval $I=[0, 1)^{d}$, that is $M^{p,q}(I)$ and $W^{p,q}(I)$ )  in  \cite[Appendix B]{Tbenyi} (see also \cite[Remark 3.3]{Tbenyi}) via establishing the equivalent of norms: 
\begin{equation}\label{t2c}
\|f\|_{M^{p,q}(\mathbb T^d)} \asymp \|f\|_{M^{p,q}(I)} \ \text{and} \  \|f\|_{W^{p,q}(\mathbb T^d)} \asymp \|f\|_{W^{p,q}(I)},
\end{equation}
where $ 1\leq p , q \leq \infty.$

\begin{Proposition} \label{Prop3.3}
Suppose that $T_{F}$ is the composition operator associated to a complex function $F$ on $\mathbb C,$ and $1\leq q < 2.$ If  $T_{F}$ takes $W^{p,1}(\mathbb R^{d})$ to $W^{p,q}(\mathbb R^{d})$, then $T_{F}$ takes  $A^{1}(\mathbb T^{d})$ to $A^{q}(\mathbb T^{d}).$
\end{Proposition}
\begin{proof}
Let $ f\in A^1(\mathbb T ^d).$ Then $f^{\ast}(x)=f(e^{2 \pi ix_1},..., e^{2 \pi ix_d})$ is a periodic function 
on $\mathbb R^d$ with absolutely convergent Fourier series $$f^*(x)=  \sum_{m \in \mathbb Z^d} \widehat{f} (m)\, e^{2 \pi i m \cdot x} .$$ Choose $g\in C_{c}^{\infty}(\mathbb R^d)$ such that $g\equiv1$ on $Q_d=[0, 1)^d.$ Then  we claim that
$gf^{\ast}\in W^{1,1}(\mathbb R^d) \subset W^{p,1}(\mathbb R^d).$ Once the claim is assumed, by hypothesis, we have
\begin{eqnarray}\label{rh}
F(g  f^{\ast})\in W^{p,q}(\mathbb R^d).
\end{eqnarray}
Note that if $z\in \mathbb T^d$, then $z= (e^{2\pi ix_1},... , e^{2\pi ix_d})$ for some $x=(x_1,... ,x_d) \in Q_d$, hence 
\bea\label{cnct} F(f(z))=F( f^*(x))= F(g  f^{\ast}(x)), ~ \mbox{for}~ x \in Q_d.\eea
Now if $\phi\in C_c^\infty(\mathbb T^d)$, then $ g \phi^*$ is a compactly supported smooth function on $\R^d$. Also $\phi(z) = g(x) \phi^*(x)$ for every $x \in Q_d$, as per the notation above and hence 
\bea \label{3.2} \phi(z) F(f)(z)= g(x)\phi^*(x) F(gf^*)(x),\eea for some $ x \in Q_d$.

By \eqref{3.2}, Proposition \ref{B1}, \eqref{t2c}, \eqref{litwd}, and Proposition \ref{ap}\eqref{aws},   we obtain
\begin{eqnarray}
\|\phi F(f)\|_{ A^{q}(\mathbb T^d)} & = & \| g \phi^* F(gf^{\ast})\|_{ A^{q}(\mathbb T^d)}\nonumber \\
& \asymp & \|g\phi^* F(gf^*)\|_{W^{p,q}(\mathbb T^{d})} \nonumber\\
& \asymp & \|g\phi^* F(gf^*)\|_{W^{p,q} (Q_d)}  \nonumber \\
& \lesssim & \|g \phi^{\ast} F(gf^*)\|_{W^{p,q}} \nonumber \\
& \lesssim & \|g\phi^*\|_{W^{\infty, 1}} \|F(gf^*)\|_{W^{p,q}},\nonumber 
\end{eqnarray}
which is finite for every smooth cutoff function $\phi$ supported on $Q_d$ in view of Lemma \ref{pl} \eqref{sce}, and \eqref{rh}.  Now by compactness of $\T^d$, a partition of unity argument shows that $F(f) \in A^{q}(\mathbb T^d)$.

To complete the proof, we need to prove the claim. By Lemma \ref{pl}\eqref{fti}, it is enough to show that 
$\widehat{gf^{\ast}}= \widehat{g} \ast \widehat{f^{\ast}} \in W^{1,1}(\mathbb R^d).$ 
We put, $\mu=\sum_{k\in \mathbb Z^{d}}c_{k}\delta_{k},$ where $c_{k}=\widehat{f}(k)$ and  $\delta_{k}$ is the unit Dirac mass at $k.$  We note that, $\mu$ is a complex Borel measure on $\mathbb R^{d},$ and the  total variation of $\mu, $ that is, $\|\mu\|= |\mu|(\mathbb R^{d})= \sum_{k\in \mathbb Z^{n}} |c_{k}|$ is finite.
We compute the Fourier-Stieltjes transform of $\mu:$
\begin{eqnarray}
\widehat{\mu}(y)  & = & \int_{\mathbb R^{d}} e^{-2\pi ix\cdot y} d\mu(x)\nonumber\\
& = & \int_{\mathbb R^{d}} e^{-2\pi ix\cdot y}(\sum_{k\in \mathbb Z^{d}} c_{k}d\delta_{k}(x))\nonumber \\
& = & \sum_{k\in \mathbb Z^{d}} c_{k} \int_{\mathbb R^{d}} e^{2\pi ix\cdot y} d\delta_{k} (x) \nonumber \\
& = & f^{\ast}(-y).\nonumber
\end{eqnarray}
So, $$\widehat{f^\ast} = \mu = \sum_{m \in \mathbb Z^d} \widehat{f} (m) \, \delta_{m}.$$ It follows that 
$$  \widehat{g} \ast \widehat{f^{\ast}}=  \sum_{m \in \mathbb Z^d} \widehat{f} (m)  \,  \widehat{g} \ast \delta_{m}
= \sum_{m \in \mathbb Z^d} \widehat{f} (m) \,  T_{m} \widehat{g} .$$
Since the translation operator $T_m$ is an isometry on $W^{1,1}(\mathbb R^d),$ it follows that the above series is absolutely convergent in $W^{1,1}(\mathbb R^d)$, and hence $\widehat{gf^{\ast}} \in W^{1,1}(\mathbb R^d)$ as claimed.
\end{proof}

\begin{proof}[Proof of Theorem \ref{MT} \eqref{nc1}]
If  $T_{F}$ takes $W^{p,1}(\mathbb R^{d})$  to $W^{p,q}(\mathbb R^{d})$, then  $T_{F}$ takes $A^{1}(\mathbb T^{d})$  to $ A^{q}(\mathbb T^{d})$  by Proposition \ref{Prop3.3}. Hence the analyticity follows from  Theorem \ref{HKKR}.

The necessity of $F(0)=0$ is obvious if $p<\infty$ and can be obtained by taking Lemma \ref{pl}\eqref{tz} into our account and testing $T_{F}$ for zero function. In fact, we can compute (see  \cite[Proof of Theorem 14]{benyi})  the STFT of  constant a function (say 1) with respect to the windowed function $g(\xi)=e^{-\pi |\xi|^2}$ can be given by
\begin{eqnarray*}\label{cstft}
|V_{g}1(x,w)|= e^{-\pi |w|^2}.
\end{eqnarray*}
From this it is clear that nonzero constant functions cannot be in $W^{p,q}(\mathbb R^d)$ if $p<\infty.$  This completes the proof of Theorem \ref{MT}\eqref{nc1} for the pair $(X,Y)= (W^{p,1}(\mathbb R^d), W^{p,q}(\mathbb R^d)).$ Taking  Proposition \ref{B1} into  account and exploiting the method  as before,   the proof of Theorem \ref{MT}\eqref{nc1}  can be obtained for the pair $(X,Y)=(M^{p,1}(\mathbb R^d), M^{p,q}(\mathbb R^d)).$
\end{proof}
\begin{proof}[Proof of Corollary \ref{op}]
The nonlinear mapping $F:\mathbb C \to \mathbb C :z\mapsto z|z|^{\alpha}$ is not real analytic on $\mathbb R^{2}$ for $\alpha \in (0, \infty) \setminus  2\mathbb N.$
\end{proof}

\section{Sufficient Conditions}
We recall that  in 1932 Wiener proved that if $F(z)=\frac{1}{z},$ then $T_{F}$ acts on $A(\mathbb T)\setminus \{0\},$ and in 1935 L\'evy generalize this result: if $F$ is real analytic on $\mathbb R^{2},$ then  $T_{F}$ acts on $A(\mathbb T).$ This is called Wiener-L\'evy \cite{wiener,levy}  theorem.  Now in this section, we shall prove Theorem \ref{MT}\eqref{sc}, and we note that our approach of proof is inspired by the Wiener-L\'evy theorem.

Unless explicitly mentioned, throughout this section we assume that  $X=M^{p,1}(\mathbb R^d)$ or $W^{p,1}(\mathbb R^d)$ with $1\leq p< \infty.$
First, we collect some technical results which should be regarded as the tool to proving  Theorem \ref{MT}\eqref{sc}. We start with following:
\begin{Definition}
Let $\phi$ be a function defined on $\mathbb R^{d}$. 
\begin{enumerate}
\item We say that $\phi$ belongs to $X$ locally at point $\gamma_{0}\in \mathbb R^{d}$ if there is a neighbourhood $V$ of $\gamma_{0}$ and a function $h\in X$ such that $\phi(\gamma)= h(\gamma)$ for every $\gamma \in V.$
\item We say that $\phi$ belongs to $X$ locally at $\infty$ if there is a compact set $K\subset \mathbb R^{d}$ and  a function $h\in X$ such that $\phi(\gamma)=h(\gamma)$ in the complement of $K.$
\end{enumerate} 
\end{Definition}
The next lemma gives the useful criterion for functions to be in $X.$
\begin{Lemma}
\label{l2g}
If $\phi$ belongs to $X$ locally at every point of  $\mathbb R^{d} \cup \{\infty \} $, then $\phi \in X.$
\end{Lemma}
To prove Lemma \ref{l2g} and for the sake of the convenience of reader first we recall following two lemmas:  

\begin{Lemma}[\textbf{The $C^{\infty}$ Urysohn Lemma}] 
\label{urysohn}
If $K\subset \mathbb R^{d}$ is comapct and $U$ is an open set containing $K,$ there exists $f\in C_{c}^{\infty}(\mathbb R^{d})$ such that $0\leq f \leq 1, f=1$ on $K,$ and $\text{supp}(f)\subset U.$ (For the proof see \cite[p. 245]{fol}). 
\end{Lemma}

\begin{Lemma}
\label{fpde}
Suppose $K\subset \mathbb R^{d}$ be compact and let $V_{1},...,V_{n}$ be open sets with $K\subset \cup_{j=1}^{n} V_{j}.$ Then there exist open sets $W_{1}, W_{2},..., W_{n}$ with $\overline{W_{j}}\subset V_{j}$ and $K\subset \cup_{j=1}^{n}W_{j}.$
\end{Lemma}
\begin{proof}
For each $\epsilon > 0$ let $V_{j}^{\epsilon}$ be the set of points in $V_{j}$ whose distance from $\mathbb R^{d}\setminus V_{j} $ is greater than $\epsilon.$ Clearly $V_{j}^{\epsilon}$ is open and $\overline{V_{j}^{\epsilon}}\subset V_{j}.$ It follows that $K\subset \cup_{1}^{n}V_{j}^{\epsilon}$ if $\epsilon$ is sufficiently small.
\end{proof}
Now we shall  prove Lemma \ref{l2g}:
\begin{proof}[Proof of Lemma \ref{l2g}]
Suppose first that $\phi$ has a compact support $K.$ By hypothesis, it follows that, for any $\gamma \in K,$ there is a neighbourhood of $\gamma,$ say $V_{\gamma},$ and $h_{\gamma} \in W^{p,1}(\mathbb R^{d})$ such that, $\phi(x)= h_{\gamma}(x)$ for all $x\in V_{\gamma}.$ Next, we observe that, $\{V_{\gamma}: \gamma \in K \}$ forms an open cover of $K,$ since $K$ is compact, 
there  exist  open sets $V_{\gamma_{1}},..., V_{\gamma_{n}}$ and functions $h_{1},...,h_{n}\in W^{p,1}(\mathbb R^{d})$ such that $\phi= h_{i}$ in $V_{\gamma_{i}}$  and $V_{\gamma_{1}}\cup V_{\gamma_{2}} \cup...\cup V_{\gamma_{n}}$ covers $K,$ 
that is, $K\subset \cup_{j=1}^{n} V_{\gamma_{j}}.$
Then by Lemma \ref{fpde}, we have 

(i) open sets $W_{1},..., W_{n}$ with compact closures $\overline{W_{j}} \subset V_{\gamma_{j}}$ such that $W_{1}\cup...\cup W_{n}$ covers $K,$ that is, $K\subset \cup_{j=1}^{n}W_{j};$
and by Lemma \ref{urysohn}, we get

(ii) functions $k_{j}\in W^{p,1}(\mathbb R^{d})$ such that $k_{j}=1$ on $\overline{W_{j}}$ and $k_{j}=0$ out side $V_{\gamma_{j}}.$\\
\noindent
Now, by using (i) and (ii), we have,  $\phi(x) k_{j}(x)=h_{j}(x)k_{j}(x),$ for all $x\in \mathbb R^{d}$ and by Proposition \ref{ap}, we get, $h_{j}k_{j}\in W^{p,1}(\mathbb R^{d})$, and so $\phi k_{j}\in W^{p,1}(\mathbb R^{d}),$  
for $1\leq j \leq n.$ Therefore, if we  put
\begin{eqnarray}
\label{bs}
\psi=\phi\{1-(1-k_{1}) (1-k_{2})...(1-k_{n})\}
\end{eqnarray}
it follows that $\psi\in W^{p,1}(\mathbb R^{d}).$ 

The multiplier of $\psi$ in \eqref{bs} is 1 whenever one of $k_{i}$ is 1, and this happens at every point of $K;$ out side $K,$ $\phi=0;$  hence $\psi=\phi,$ and thus $\phi \in W^{p,1}(\mathbb R^{d}).$

In the general case, $\phi$ belongs to $W^{p,1}(\mathbb R^{d})$ locally at $\infty,$  so that there is a function $g\in W^{p,1}(\mathbb R^{d})$ which coincides with $\phi$ outside some compact subset of $\mathbb R^{d}.$ Then $\phi- g$ has compact support and belongs to $W^{p,1}(\mathbb R^{d})$ locally at every point of $\mathbb R^{d};$ by the first case, $\phi- g\in W^{p,1}(\mathbb R^{d}),$ and so $\phi \in W^{p,1}(\mathbb R^{d}).$ This completes the proof if $X=W^{p,1}(\mathbb R^d).$ The case $X=M^{p,1}(\mathbb R^d)$ can be obtained similarly.
\end{proof} 

We denote by  $X_{loc}$ the space functions that are locally in $X$ at each $\gamma_0\in \mathbb R^d.$
\begin{Lemma}[\cite{dgb-pkr},p.634] Let  $f$ be a function defined on $\mathbb R^d.$ 
\begin{enumerate}
\item $f\in X_{loc}$ if and only if $\phi f \in X$ for all $\phi \in C_c^{\infty}(\mathbb R^d).$
\item $f$ belongs to $X$ locally at $\infty$ if and only if there exists $\phi \in C_{c}^{\infty}(\mathbb R^d)$ such that $(1-\phi)f\in X.$
\end{enumerate}
\end{Lemma}
\begin{Proposition}  (\cite [Proposition 3.4]{dgb-pkr}) \label{small} 
Let $f\in M^{1,1}(\R^d )$, $x_0 \in \R^d$ and $\epsilon >0$. Then there exists a $\phi \in C_c^\infty(\R^d)$ such that  $\| \phi  \left[ f -f(x_0)  \right] \|_{M^{1,1}} < \epsilon$.  The function $\phi $ can be chosen so that $\phi \equiv 1$ in  some neighbourhood of $x_0$. 
\end{Proposition}
\begin{Proposition} Let $f\in X, x_0 \in \R^d$ and $\epsilon >0$. Then there exists a $\Phi \in C_c^\infty(\R^d)$ such that  $\| \Phi  \left[ f -f(x_0)  \right] \|_{X} < \epsilon$.  The function $\phi $ can be chosen so that $\Phi \equiv 1$ in  some neighbourhood of $x_0$. 
\end{Proposition}
\begin{proof}
Let $f\in X.$ Choose $\psi \in C_{c}^{\infty}(\mathbb R^d)$ such that $\psi\equiv 1$ in some neoghbourhood of $x_0.$  In view of $M^{1,1}(\mathbb R^d)=W^{1,1}(\mathbb R^d)$  and   Proposition \ref{ap},  we have 
\begin{eqnarray*}
\|\psi f\|_{M^{1,1}} &  \leq &  \|f\|_{Y} \|\psi\|_{M^{1,1}}\\
& \lesssim & \|f\|_{X}< \infty 
\end{eqnarray*}
where $Y=M^{\infty,1}$ if $X=M^{p,1}$ and $Y=W^{\infty,1}$ if $X=W^{p,1}.$ 
Thus $h:= \psi f \in M^{1,1}(\mathbb R^d).$
We  can now apply  Proposition \ref{small} for $h\in M^{1,1}(\mathbb R^d):$ given $\epsilon'>0,$     there exists $\phi\in C_{c}^{\infty}(\mathbb R^d)$ such that $$\|\phi (h-h(x_0))\|_{M^{1,1}}<\epsilon'$$ and $\phi \equiv 1$ in some neighbourhood of $x_0.$  Now define $\Phi (x)= \psi (x) \phi(x)$ for all $x\in \mathbb R^d.$  Note that $\Phi \in C_{c}^{\infty}(\mathbb R^d)$ and $\Phi \equiv 1$ on some neighbourhood of $x_0.$
By definition of $\Phi$ and Lemma \ref{pl}, we have
\begin{eqnarray*}
\|\Phi(f-f(x_0))\|_{X} & = & \|\phi(h-h(x_0))\|_{X} \\
& \leq & C \|\phi(h-h(x_0))\|_{M^{1,1}}< C\epsilon'.
\end{eqnarray*}
Taking $\epsilon'=\epsilon/C, $ we get $ \|\Phi(f-f(x_0))\|_{X} < \epsilon.$ This completes the proof.
\end{proof}
\begin{Proposition}\label{kp} Let $ f\in X,$ and $\epsilon>0.$ There also exists a $\psi \in C_c^\infty(\R^d)$ such that $\| (1-\psi) f \|_{X}<\epsilon$.
\end{Proposition}
\begin{proof} Let $f\in X,$ and $\epsilon'>0.$
By Lemma \ref{pl} \eqref{da},  there exists $g\in \mathcal{S}(\mathbb R^d)$ such that   
\begin{eqnarray}\label{lm1}
\|f-g\|_{X}< \epsilon'.
\end{eqnarray}
We recall the fact that for any  $g\in \mathcal{S}(\mathbb R^d),$  there exists $\lambda_0\in (0,1)$ such that
\begin{eqnarray}\label{lm2}
\|(1-\phi_{\lambda})g\|_{M^{1,1}} < \frac{\epsilon}{2}
\end{eqnarray}
for any $\lambda \in (0, \lambda_0),$ where $\phi_{\lambda}(x)= \phi(\lambda x)\in C_c^{\infty}(\mathbb R^d)$,  (see for instance the proof of \cite[Proposition 3.14]{dgb-pkr}).
We define $\psi  \in C_{c}^{\infty}(\mathbb R^d)$ such that $\psi(x):=\phi(\lambda x)$  where $\lambda \in (0, \lambda_0).$ 
By Lemma \ref{pl} and \eqref{lm2}, we  have 
\begin{eqnarray*}
\|(1-\psi)f\|_{X}  & \leq &  \|(1- \psi) (f-g)\|_{X}+ \|(1-\psi)g\|_{X}\\
& = & \|f-g-\psi (f-g)\|_{X} + \|(1-\phi_{\lambda})g\|_{X}\\
& \leq & \|f-g\|_{X} + C \|\psi\|_{Y}\|f-g\|_{X} + \|(1-\phi_{\lambda})g\|_{M^{1,1}}\\
& \leq &  (1+ C \|\phi_{\lambda}\|_{Y}) \|f-g\|_{X}+ \frac{\epsilon}{2} 
\end{eqnarray*}
 where $Y=M^{\infty,1}$  if $X=M^{p,1}$ and $Y=W^{\infty,1}$ if $X=M^{p,1}.$  By Lemma \ref{pl} \eqref{rmw} and Lemma \ref{pl}\eqref{dp}, we have
we have $\|\phi_{\lambda}\|_{Y} \lesssim \|\phi_{\lambda}\|_{M^{\infty,1}} \lesssim \|\phi\|_{M^{\infty,1}}$. Using this, we have
\begin{eqnarray}\label{lm3}
\|(1-\psi)f\|_{X} \leq (1+C'\|\phi\|_{M^{\infty,1}}) \|f-g\|_{X} +\frac{\epsilon}{2}.
\end{eqnarray}
Taking $\epsilon'=\frac{\epsilon}{2 (1+C'\|\phi\|_{M^{\infty,1}})},$ and using \eqref{lm1} and \eqref{lm2}, we obtain that $\|(1-\psi)f\|_{X}< \epsilon.$
\end{proof}
\begin{proof}[Proof of Theorem \ref{MT}\eqref{sc}]
Write $f= f_1+ if_2 \in X$, where $f_1$ and $f_2$ are real functions, and with an abuse of notation, we write $F(f) = F (f_1, f_2)$.
To show that $F(f)$ is in $X$,  enough to show, in view of Lemma \ref{l2g} that 
$F(f) \in X_{loc}$ and $F(f)$ belongs to $X$  locally at $\infty.$
First we show that $F(f) \in X_{loc}.$ 
Fix $x_{0} \in \R^d $ and put $f(x_0)= s_{0} + it_{0}$. 
Since $F$ is real analytic at $(s_0, t_0)$, there exists a $\delta >0 $ such that $F$ has the power  series expansion
\begin{eqnarray} \label{3.4}
\label{mc}
F(s,t)= F(s_{0}, t_{0}) + \sum_{m,n=0}^{\infty} a_{mn} (s-s_{0})^{m} (t-t_{0})^{n},~ ~(a_{00}=0)
\end{eqnarray}
which converges absolutely for $|s-s_{0}|\leq \delta, |t-t_{0}|\leq \delta.$
Then 
\bea
\label{mc2}
F(f_1(x),f_2(x)) & = & F(s_{0}, t_{0})   \nonumber \\
&&+ \sum_{(m,n)\neq (0,0) } a_{mn} [f_1(x)-f_1(x_{0}) ]^{m} [f_2(x)-f_2(x_{0})]^{n}
\eea
whenever the series converges. 

Note that both $f_1$ and $f_2$ are in $X$, being the real and imaginary part of $f$. 
Hence in view of  Proposition \ref{small},  we can find a $\phi \in C_c^\infty(\R^d)$, such that 
$\phi \equiv 1$ near $x_0$ and $\| \phi [ f_i - f_i(x_0) ]  \|_{X}< \delta $, for $i=1,2$.
Now consider the function $G$ on $\R^d$ defined by 
\bea
G(x) &=& \phi(x) \, F(s_{0}, t_{0}) \nonumber\\
&& + \sum_{(m,n)\neq (0,0) } a_{mn} \left( \phi(x)[ f_1(x)-f_1(x_{0}) ] \right) ^{m} \left( \phi(x) [f_2(x)-f_2(x_{0}) ] \right)^{n}.\eea

Since  $\| \varphi [f_i -f_i(x_0)] \|_{X}< \delta $, for $i=1,2$ and in view of Proposition \ref{ap}, we see that the above series is absolutely convergent in $X.$
Also since $\phi \equiv 1$ in some neighbourhood of $x_0$, it follows that $G \equiv F(f)$ in some  neighbourhood of $x_0$.
Since $x_0$ is arbitrary, this shows that $F(f) \in X_{loc}.$

To show that $F(f)\in X$ locally at infinity, we take $(s_0,t_0)=(0,0)$ in equation (\ref{3.4}).  Since $F({ 0})=0,$  the expansion \eqref{mc2} now becomes
\Bea
\label{mc}
F(f_1(x),f_2(x)) & = & \sum_{(m,n)\neq (0,0) } a_{mn} \, [f_1(x)] ^{m} \, [f_2(x)]^{n},
\Eea  
whenever the series converges.

By Proposition \ref{kp}, we have $\|(1-\psi) f_i \|_{X} < \delta$, for $i=1,2$ for some $\psi \in C_c^\infty(\R^d)$. 
 Now consider the function $H$ defined by
\Bea
H(x)= \sum_{(m,n)\neq (0,0) } a_{mn} \, [(1-\psi (x))f_1(x) ]^m \, [(1-\psi (x))f_2(x) ]^n.\Eea

The above series is absolutely convergent in $X$, in view of the above norm estimates, hence
$H \in X$.
Also since $\psi$ is compactly supported, $1- \psi \equiv 1 $ in the complement of a large ball centered at the  origin, hence $H =F(f)$ in the compliment of a compact set. This shows that $F(f)$ belongs to $X$  locally at infinity. 
\end{proof}
Finally, in this section we note 
\begin{Theorem}
Suppose that $T_{F}$ is the composition operator associated to a complex function $F$ on $\mathbb C,$ and $1\leq p \leq \infty.$ If $F$ is a real entire given by $F(s,t)=\sum_{m,n=0}^{\infty} a_{mn}s^{m}t^{n}$ with $F(0)=0$, then $T_{F}$ acts on $W^{p,1}(\mathbb R^{d}).$ In particular, we have 
$$\|T_{F}(f)\|_{W^{p,1}}\leq \sum_{m,n=0}^{\infty} |a_{mn}|\|f\|_{W^{p,1}}^{m+n}, (f=f_{1}+if_{2}).$$
\end{Theorem}
\begin{proof}
Let $f\in W^{p,1}(\mathbb R^{d})$ with $f_{1}= \frac{f+\bar{f}}{2}$ and $f_{2}=\frac{f-\bar{f}}{2i}$. Then $f_{1}, f_{2}\in W^{p,1}(\mathbb R^{d})$ and so $f_{1}^{m}, f_{2}^{n}\in W^{p,1}(\mathbb R^{d})$ by Proposition \ref{ap}. Since the series $\sum_{m,n=0}^{\infty} a_{mn}s^{m}t^{n}$, converges absolutely for all $(s,t),$  the series $\sum_{n,m=0}^{\infty}a_{mn} f^{m}_{1} f^{n}_{2}$ is converges in the norm of $W^{p,1}(\mathbb R^{d})$; and its sum is $F(f)=\sum_{n,m=0}^{\infty}a_{mn}f_{1}^{m}f_{2}^{m}$; and hence 
$$\|F(f)\|_{W^{p,1}}\leq \sum_{ m, n=0}^{\infty} |a_{mn}|\cdot \|f_{1}\|^{m}_{W^{p,1}} \|f_{2}\|
^{n}_{W^{p,1}}.$$\end{proof}
\section{Concluding Remarks}
By the frequency-uniform localization (see \cite[Chapter 6]{baob}) techniques, the modulation and Wiener amalgam spaces can be viewed as a Besov/Lizorkin-Triebel type space associated with a uniform decomposition (see \cite{tribel, bao}). We note that in the last two decades composition operators have been studied (by Bourdaud, Sickel et al.) extensively on Besov and Lizorkin-Triebel spaces.  We cannot hope to acknowledge here all those who made this story of a composition operator successful. And we just refer  to the enlightening survey article \cite{bs} by Bourdaud-Sickel and the references therein. Recently   \cite{rrs, kst, nonom} some progress has been made for a composition operator  on weighted modulation spaces.  But we believe, yet we have very little  information for composition operators on  modulation and Wiener amalgam spaces. Specifically, we note following:
\begin{enumerate}
\item Feichtinger \cite{HG4} have established the basic properties  Wiener amalgam spaces and modulation spaces  on locally compact groups. It would be interesting to investigate the analogue of Theorems \ref{MT}  for locally compact groups.

\item  We have answered the problem stated in introductory paragraph in a few specific cases   (Theorems \ref{MT}). What about the remaining cases? 
\item  It would be interesting to find necessary and sufficient conditions on $F:\mathbb C\to \mathbb C$ such that  the composition operator is bounded on modulation/Wiener amalgam spaces, that is,  $\|F\circ f \|_{X} \lesssim \|f\|_{X}$ for $X=M^{p,q}(\mathbb R^d)$ or $W^{p,q}(\mathbb R^d)$ ($p=q\neq 2$).
\end{enumerate}
\noindent
{\textbf{Acknowledgment}:} 
The author wishes to thank Prof. P. K. Ratnakumar for
suggesting to look at this problem and encouragement on the subject of this paper. The author  is  thankful to   IUSTF and Indo-US SERB and DST-INSPIRE, and TIFR CAM for the  support.  The author is very grateful to Professor Kasso Okoudjou for his  hospitality and   arranging  excellent research facilities at   the University of Maryland.

\end{document}